\pdfoutput=1
\documentclass{article}

\usepackage{microtype}
\usepackage{graphicx}
\usepackage{subfigure}
\usepackage{booktabs} 

\usepackage{natbib}
\usepackage{algorithm}
\usepackage{algorithmic}
\usepackage{amsmath,amsthm,amsfonts}
\usepackage{hyperref}
\usepackage{color}
\usepackage{mathrsfs}
\usepackage{enumitem}
\usepackage{bm}
\usepackage{multirow}
\usepackage{dsfont}
\usepackage{subfiles}
\usepackage{xspace}

\usepackage{makecell}
\usepackage{graphicx}
\usepackage{subfigure}
\usepackage{caption}
\usepackage{thmtools}
\usepackage{thm-restate}

\newcommand{\ingd}{\textsc{Ingd}\xspace}

\newcommand{\grad}{\nabla}

\newcommand{\iprod}[2]{\langle #1, #2 \rangle}

\newcommand{\cO}{\mathcal{O}}

\newcommand{\R}{\mathbb{R}}

\newcommand{\E}{\mathbb{E}}

\newcommand{\G}{\mathcal{G}}
\newcommand{\Y}{\mathcal{Y}}

\newcommand{\fF}{\mathfrak{F}}

\newtheorem{theorem}{Theorem}
\newtheorem{lemma}[theorem]{Lemma}

\newtheorem{remark}[theorem]{Remark}

\theoremstyle{definition}
\newtheorem{definition}{Definition}

\newtheorem{assumption}{Assumption}
\newtheorem{proposition}[theorem]{Proposition}

\usepackage[accepted]{icml2020}

\icmltitlerunning{Complexity of Finding Stationary Points of
Nonsmooth Nonconvex Functions}
\begin{document}

\twocolumn[
\icmltitle{Complexity of Finding Stationary Points of
Nonsmooth Nonconvex Functions}




\begin{icmlauthorlist}
\icmlauthor{Jingzhao Zhang}{mit}
\icmlauthor{Hongzhou Lin}{mit}
\icmlauthor{Stefanie Jegelka}{mit}
\icmlauthor{Suvrit Sra}{mit}
\icmlauthor{Ali Jadbabaie}{mit}

\end{icmlauthorlist}

\icmlaffiliation{mit}{Massachusetts Institute of Technology}

\icmlcorrespondingauthor{Jingzhao Zhang}{jzhzhang@mit.edu}

\icmlkeywords{Optimization, Non-convex optimization}

\vskip 0.3in
]



\printAffiliationsAndNotice{}  

\begin{abstract}
We provide the first \emph{non-asymptotic} analysis for  finding stationary points of nonsmooth, nonconvex functions. In particular, we study the class of Hadamard semi-differentiable functions, perhaps the largest class of nonsmooth functions for which the chain rule of calculus holds. This class contains examples such as ReLU neural networks and others with non-differentiable activation functions. We first show that finding an $\epsilon$-stationary point with first-order methods is impossible in finite time. We then introduce the notion of  \emph{$(\delta, \epsilon)$-stationarity}, which allows for an $\epsilon$-approximate gradient to be the convex combination of generalized gradients evaluated at points within distance $\delta$ to the solution. We propose a series of randomized first-order methods and analyze their complexity of finding a $(\delta, \epsilon)$-stationary point. Furthermore, we provide a lower bound and show that our stochastic algorithm has min-max optimal dependence on $\delta$. Empirically, our methods perform well for training ReLU neural networks.

\end{abstract}

\vspace*{-15pt}
\section{Introduction}
\vspace*{-3pt}
Gradient based optimization underlies most of machine learning and it has attracted tremendous research attention over the years. While non-asymptotic complexity analysis of gradient based methods is well-established for convex and \emph{smooth} nonconvex problems, little is known for nonsmooth nonconvex problems. We summarize the known rates (black) in Table~\ref{tab:summary} based on the references~\citep{nesterov2018lectures, carmon2017lower, arjevani2019lower}.


\begin{table}[h]
\caption{
When the problem is nonconvex and nonsmooth, finding a  $\epsilon$-stationary point is intractable, see Theorem~\ref{thm:lower}. Thus we introduce a refined notion,  $(\delta,\epsilon)$-stationarity, and provide non-asymptotic convergence rates for finding $(\delta,\epsilon)$-stationary point.
}
\label{tab:summary}
\vskip 0.15in
\begin{center}
\begin{small}
\begin{sc}
\begin{tabular}{lccr}
\toprule
 \textbf{Deterministic rates} & Convex & Nonconvex \\
\midrule
L-smooth    & $\cO(\epsilon^{-0.5})$ & $\cO(\epsilon^{-2})$\\
L-Lipschitz    & $\cO(\epsilon^{-2})$ &  \textcolor{blue}{$\tilde{\cO}(\epsilon^{-3}\delta^{-1})$} \\
\end{tabular}
\vskip 0.15in
\begin{tabular}{lccr}
\toprule
 \textbf{Stochastic rates} & Convex & Nonconvex \\
\midrule
L-smooth    & $\cO(\epsilon^{-2})$ & $\cO(\epsilon^{-4})$\\
L-Lipschitz    & $\cO(\epsilon^{-2})$ &  \textcolor{blue}{$\tilde{\cO}(\epsilon^{-4}\delta^{-1})$}  \\
\end{tabular}
\end{sc}
\end{small}
\end{center}
\hrule
\vskip -13pt
\end{table}

Within the nonsmooth nonconvex setting, recent research results have focused on asymptotic convergence analysis~\citep{benaim2005stochastic, kiwiel2007convergence, majewski2018analysis, davis2018stochastic, bolte2019conservative}. Despite their advances, these results fail to address finite-time, non-asymptotic convergence rates. Given the widespread use of nonsmooth nonconvex problems in machine learning, a canonical example being deep ReLU neural networks, obtaining a \emph{non-asymptotic} convergence analysis is an important open problem of fundamental interest. 

We tackle this problem for nonsmooth functions that are Lipschitz and directionally differentiable. This class is rich enough to cover common machine learning problems, including ReLU neural networks. Surprisingly, even for this seemingly restricted class, finding an $\epsilon$-stationary point, i.e., a point $\bar{x}$ for which $d(0, \partial f(\bar{x})) \le \epsilon$,  is intractable. In other words, no algorithm can guarantee to find an $\epsilon$-stationary point within a \emph{finite} number of iterations. 

This intractability suggests that, to obtain meaningful non-asymptotic results, we need to refine the  notion of stationarity. We introduce such a notion and base our analysis on it, leading to the following main contributions of the paper:
\begin{itemize}[leftmargin=1em]
  \setlength\itemsep{-1pt}
  \vspace*{-8pt}
\item We show that a traditional $\epsilon$-stationary point cannot be obtained in finite time (Theorem~\ref{thm:finite}). 
\item We study the notion of $(\delta, \epsilon)$-stationary points (see Definition~\ref{def:stationary}). For smooth functions, this notion reduces to usual  $\epsilon$-stationarity by setting $\delta = O(\epsilon/L)$. We provide a $\Omega(\delta^{-1})$ lower bound on the number of calls if algorithms are only allowed access to a generalized gradient oracle.
\item We propose a normalized ``gradient descent'' style algorithm that achieves $\tilde{\cO}(\epsilon^{-3}\delta^{-1})$ complexity in finding a $(\delta,\epsilon)$-stationary point in the deterministic setting. 
\item We propose a momentum based algorithm that achieves $\tilde{\cO}(\epsilon^{-4}\delta^{-1})$ complexity in finding a $(\delta,\epsilon)$-stationary point in the stochastic finite variance setting.
\end{itemize}
As a proof of concept to validate our theoretical findings, we implement our stochastic algorithm and show that it matches the performance of empirically used SGD with momentum method for training ResNets on the Cifar10 dataset.

Our results attempt to bridge the gap from recent advances in developing a non-asymptotic theory for nonconvex  optimization algorithms to settings that apply to training deep neural networks, where, due to non-differentiability of the activations, most existing theory does not directly apply.

\subsection{Related Work}\label{sec:related}
\textbf{Asymptotic convergence for nonsmooth nonconvex functions.} \citet{benaim2005stochastic} study the convergence of subgradient methods from a differential inclusion perspective; \citet{majewski2018analysis} extend the result to include proximal and implicit updates. \citet{bolte2019conservative} focus on formally justifying the back propagation rule under nonsmooth conditions. In parallel, \citet{davis2018stochastic} proved asymptotic convergence of subgradient methods assuming the objective function to be Whitney stratifiable. The class of Whitney stratifiable functions is broader than regular functions studied in~\citep{majewski2018analysis}, and it does not assume the regularity inequality (see Lemma 6.3 and (51) in~\citep{majewski2018analysis}). Another line of work \citep{mifflin1977algorithm, kiwiel2007convergence, burke2018gradient} studies convergence of gradient sampling algorithms. These algorithms assume a deterministic generalized gradient oracle. Our methods draw intuition from these algorithms and their analysis, but are non-asymptotic in contrast.

\textbf{Structured nonsmooth nonconvex problems.} Another line of research in nonconvex optimization is to exploit structure: \citet{duchi2018stochastic, drusvyatskiy2019efficiency, davis2019stochastic} consider the composition structure $f\circ g$ of convex and smooth functions;  \citet{bolte2018first,zhang2018convergence,beck2020convergence} study composite objectives of the form $f+g$ where one function is differentiable or convex/concave. With such structure, one can apply proximal gradient algorithms if the proximal mapping can be efficiently evaluated. However, this usually requires weak convexity, i.e., adding a quadratic function makes the function convex, which is not satisfied by several simple functions, e.g., $-|x|$. 

\textbf{Stationary points under smoothness.} When the objective function is smooth, SGD finds an $\epsilon$-stationary point in $O(\epsilon^{-4})$ gradient calls~\citep{ghadimi2013stochastic}, which improves to $O(\epsilon^{-2})$ for convex problems. Fast upper bounds under a variety of settings (deterministic, finite-sum, stochastic) are studied in \citep{carmon2018accelerated, fang2018spider,zhou2018stochastic,nguyen2019optimal,allen2018make,reddi2016stochastic}. More recently, lower bounds have also been developed \citep{carmon2017lower,drori2019complexity,arjevani2019lower,foster2019complexity}. When the function enjoys high-order smoothness, a stronger goal is to find an approximate second-order stationary point and could thus escape saddle points too. Many methods focus on this goal \citep{ge2015escaping, agarwal2017finding,jin2017escape,daneshmand2018escaping, fang2019sharp}. 




\section{Preliminaries}
In this section, we set up the notion of generalized directional derivatives that will play a central role in our analysis. Throughout the paper, we assume that the nonsmooth function $f$ is $L$-Lipschitz continuous (more precise assumptions on the function class are outlined in \S\ref{sec:funct-class}).

\subsection{Generalized gradients}
We start with the definition of generalized gradients, following~\citep{clarke1990optimization}, for which we first need:
\begin{definition}
  Given a point $x \in \R^d$, and direction $d$, the \textbf{\emph{generalized directional derivative}} of $f$ is defined as 
\begin{align*}
    f^\circ(x;d) := \limsup_{y\to x, t \downarrow 0} \tfrac{f(y+td)- f(y)}{t}.
\end{align*}
\end{definition}
\begin{definition}
The \textbf{\emph{generalized gradient}} of $f$ is defined as 
\begin{align*}
    \partial f(x) := \{g \mid \iprod{g}{d} \leq f^\circ(x, d),\  \forall d \in \mathbb{R}^d\}.
\end{align*}
\end{definition}
We recall below the following basic properties of the generalized gradient, see e.g.,~\citep{clarke1990optimization} for details.

\begin{proposition}[Properties of generalized gradients]\leavevmode\\[-18pt]
\begin{enumerate}
  \setlength{\itemsep}{2pt}
    \item $\partial f(x)$ is a nonempty, convex compact set. For all vectors $g \in \partial f(x)$, we have $\|g\| \le L$.
    \item $f^\circ(x; d) = \max\{\iprod{g}{d} \mid g \in \partial f(x)\}$.
    \item $\partial f(x)$ is an upper-semicontinuous set valued map. 
    \item $f$ is differentiable almost everywhere (as it is $L$-Lipschitz); let $\mathrm{conv}(\cdot)$ denote the convex hull, then
\begin{align*}
    \partial f(x) = \text{conv}\bigl(\bigl\{g | g = \lim_{k\to \infty}\grad f(x_k),\ x_k \to x\bigr\}\bigr).
\end{align*} 
    \item Let $B$ denote the unit Euclidean ball. Then, 
      $$\partial f(x) = \cap_{\delta > 0}\cup_{y \in x+\delta B} \partial f(y).$$
    \item For any $y, z$, there exists $\lambda \in (0, 1)$ and $g \in \partial f(\lambda y + (1-\lambda)z)$ such that $f(y) - f(z) = \iprod{g}{y - z}$.
\end{enumerate}
\end{proposition}
\subsection{Directional derivatives}
Since general nonsmooth functions can have arbitrarily large variations in their ``gradients,'' we must  restrict the function class to be able to develop a meaningful complexity theory. We show below that directionally differentiable functions match this purpose well.

\begin{definition}\label{def:direction}
A function $f$ is called \textbf{\emph{directionally differentiable in the sense of Hadamard}} (\emph{cf.} \citep{sova1966, shapiro1990concepts}) if for any mapping $\varphi: \mathbb{R}_+ \to X$ for which $\varphi(0) = x$ and $\lim_{t\to 0^+}\frac{\varphi(t) - \varphi(0)}{t}=d$, the following limit exists:
\begin{align}
    f'(x;d) = \lim_{t \to 0^+}\tfrac{1}{t}(f(\varphi(t)) - f(x)).
\end{align}
\end{definition}
In the rest of the paper, we will say a function $f$ is \textbf{directionally differentiable} if it is directionally differentiable in the sense of Hadamard at all $x$. 

This directional differentiabilility is also referred to as Hadamard semidifferentiability in~\citep{delfour2019introduction}. Notably, such directional differentiability is satisfied by most problems of interest in machine learning. It includes functions such as  $f(x) = -|x|$ that \emph{do not} satisfy the so-called regularity inequality (equation (51) in~\citep{majewski2018analysis}).
Moreover, it covers the class of semialgebraic functions, as well as o-minimally definable functions (see Lemma 6.1 in \citep{coste2000introduction}) discussed in \citep{davis2018stochastic}. Currently, we are unaware whether the notion of Whitney stratifiability (studied in some recent works on nonsmooth  optimization) implies directional differentiability.

A very important property of directional differentiability is that it is preserved under composition.
\begin{lemma}[Chain rule]\label{lemma:chain-rule}
Let $\phi$ be Hadamard directionally differentiable at $x$, and $\psi$ be Hadamard directionally differentiable at $ \phi(x)$. Then the composite mapping $\psi\circ \phi$ is Hadamard directionally differentiable at $x$ and
\begin{equation*}
(\psi \circ \phi)'_x = \psi'_{\phi(x)} \circ \phi'_x.
\end{equation*}
\end{lemma}
A proof of this lemma can be found in \citep[Proposition 3.6]{shapiro1990concepts}. As a consequence, any neural network function composed of directionally differentiable functions, including ReLU/LeakyReLU, is directionally differentiable.

Directional differentiability also implies key properties useful in the analysis of nonsmooth problems. In particular, it enables the use of (Lebesgue) path integrals as follows.
\begin{lemma}\label{lemma:lebesgue-integral}
Given any $x,y$, let $\gamma(t) =  x + t(y-x)$, $t \in [0, 1]$. If $f$ is directionally differentiable and Lipschitz, then 
\begin{align*}
    f(y) - f(x) 
    &= \int_{[0,1]}f'(\gamma(t); y - x) dt.
\end{align*}
\end{lemma}
The following important lemma further connects directional derivatives with generalized gradients.
\begin{lemma}\label{lemma:directional-generalized}
Assume that the directional derivative exists. For any $x, d$, there exists $g \in \partial f(x)$ \text{s.t.} $\iprod{g}{d} = f'(x;d)$.
\end{lemma}

\subsection{Nonsmooth function class of interest}
\label{sec:funct-class}
Throughout the paper, we focus on the set of Lipschitz, directionally differentiable and bounded (below) functions:
\begin{align}\label{eq:fclass}
    \mathcal{F}(\Delta, L) := \{f | & f \text{ is $L$-Lipschitz}; \nonumber\\
    & f \text{ is directionally differentiable};\nonumber\\
    & f(x_0) - \inf_x f(x) \le \Delta\},
\end{align}
where a function $f:\mathbb{R}^n \to \mathbb{R}$ is $L-$Lipschitz if
\begin{align*}
    |f(x) - f(y)| \le L\|x-y\|, \forall \ x, y \in \mathbb{R}^n.
\end{align*}
As indicated previously, ReLU neural networks with bounded weight norms are included in this function class. 

\section{Stationary points and oracles}
We now formally define our notion of stationarity and discuss the intractability of the standard notion. Afterwards, we formalize the optimization oracles and define measures of complexity for algorithms that use these oracles.

\subsection{Stationary points}
With the generalized gradient in hand, commonly a point is called stationary if $0 \in \partial f(x)$ \cite{clarke1990optimization}. A natural question is, what is the necessary complexity to obtain an \emph{$\epsilon$-stationary point}, i.e., a point $x$ for which \[ \min\{\|g\| \mid\ g \in \partial f(x)\} \le \epsilon.\] 
It turns out that attaining such a point is intractable. In particular, there is no finite time algorithm that can guarantee $\epsilon$-stationarity in the nonconvex nonsmooth setting. We make this claim precise in our first main result.
\begin{theorem}\label{thm:finite}
Given any algorithm~$\mathcal{A}$ that accesses function value and generalized gradient of $f$ in each iteration, for any $\epsilon \in [0, 1)$ and for any finite iteration $T$, there exists $f \in \mathcal{F}(\Delta, L)$ such that the sequence $\{ x_t \}_{t\in[1,T]}$ generated by $\mathcal{A}$ on the objective $f$ does not contain any $\epsilon$-stationary point with probability more than $\tfrac{1}{2}$. 
\end{theorem} 

A key ingredient of the proof is that an algorithm~$\mathcal{A}$ is uniquely determined by $\{f(x_t), \partial f(x_t)\}_{t \in [1,T]}$, the function values and gradients at the query points. For any two functions $f_1$ and $f_2$ that have the same function values and gradients at the same set of queried points $\{x_1, ..., x_t\}$, the distribution of the iterate $x_{t+1}$ generated by $\mathcal{A}$ is identical for $f_1$ and $f_2$. However, due to the richness of the class of nonsmooth functions, we can find $f_1$ and $f_2$ such that the set of $\epsilon$-stationary points of $f_1$ and $f_2$ are disjoint. Therefore, the algorithm cannot find a stationary point with probability more than $\tfrac{1}{2}$ for both $f_1$ and $f_2$ simultaneously. Intuitively, such functions exist because a nonsmooth function could vary arbitrarily---e.g., a nonsmooth nonconvex function could have constant gradient norms except at the (local) extrema, as happens for a piecewise linear zigzag function. Moreover, the set of extrema could be of measure zero. Therefore, unless the algorithm  lands exactly in this measure-zero set, it cannot find any $\epsilon$-stationary point. 


Theorem~\ref{thm:finite} suggests the need for rethinking the definition of stationary points. Intuitively, even though we are unable to find an $\epsilon$-stationary point, one could hope to find a point that is close to an $\epsilon$-stationary point. This motivates us to adopt the following more refined notion:


\begin{definition}\label{def:stationary} 
A point $x$ is called \emph{$(\delta, \epsilon)$-stationary} if 
\begin{align*}
     d(0, \partial f(x+\delta B)) \le \epsilon,
\end{align*}
where $\partial f(x+\delta B) := \text{conv}(\cup_{y \in x + \delta B}\partial f(y))$ is the Goldstein $\delta$-subdifferential, introduced in \cite{goldstein1977optimization}. 
\end{definition}
Note that if we can find a point $y$ at most distance $\delta$ away from $x$ such that $y$ is $\epsilon$-stationary, then we know $x$ is $(\delta, \epsilon)$-stationary. However, the contrary is not true. In fact, \citep{shamir2020can} shows that finding a point that is $\delta$ close to an $\epsilon-$stationary point requires exponential dependence on the dimension of the problem.

At first glance, Definition~\ref{def:stationary} appears to be a weaker notion since if $x$ is $\epsilon$-stationary, then it is also a  $(\delta,\epsilon)$-stationary point for any $\delta \ge 0$, but not vice versa. We show that the converse implication indeed holds, assuming smoothness. 

\begin{proposition}\label{prop:equivalence} The following statements hold:
  \begin{enumerate}[label=(\roman*)]
    \setlength{\itemsep}{0pt}
    \item $\epsilon$-stationarity implies $(\delta,\epsilon)$-stationarity for any $\delta \ge 0$. 
    \item If $f$ is smooth with an $L$-Lipschitz gradient and if $x$ is ($\frac{\epsilon}{3L}$, $\frac{\epsilon}{3}$)-stationary, then $x$ is also $\epsilon$-stationary, i.e.
\[   d\left (0,  \partial f \left (x+ \tfrac{\epsilon}{3L} B \right ) \right ) \le \tfrac{\epsilon}{3} \,\, \implies \,\, \Vert \nabla f(x) \Vert \le \epsilon. \]
\end{enumerate}
\end{proposition}
Consequently, the two notions of stationarity are equivalent for differentiable functions. It is then natural to ask: \emph{does $(\delta,\epsilon)$-stationarity permit a finite time analysis?}

The answer is positive, as we will show later, revealing an intrinsic difference between the two notions of stationarity. Besides providing algorithms, in Theorem~\ref{thm:lower} we also prove an $\Omega(\delta^{-1})$ lower bound on the dependency of $\delta$  for algorithms that can only access a generalized gradient oracle.

We also note that $(\delta,\epsilon)$-stationarity behaves well as $\delta \downarrow 0$.  

\begin{lemma}\label{lemma:approx}
The set $\partial f(x+\delta B)$ converges as $\delta \downarrow 0$ as
\begin{align*}
  \lim_{\delta \downarrow 0} \partial f(x+\delta B) = \partial f(x).
\end{align*}
\end{lemma}
Lemma~\ref{lemma:approx} enables a straightforward routine for transforming non-asymptotic analyses for finding $(\delta,\epsilon)$-stationary points to asymptotic results for finding $\epsilon$-stationary points. Indeed, assume that a finite time algorithm for finding $(\delta,\epsilon)$-stationary points is provided. Then, by repeating the algorithm with decreasing $\delta_k$, (e.g., $\delta_k=1/k$), any accumulation points of the repeated algorithm is an $\epsilon$-stationary point with high probability. 




\subsection{Gradient Oracles}\label{sec:alg-class}
We assume that our algorithm has access to a generalized gradient oracle in the following manner: 
\begin{assumption}\label{assump}\ 
Given $x, d$, the oracle $\mathbb{O}(x, d)$ returns a function value $f_x$, and a generalized gradient $g_x$,
\begin{align*}
    (f_x, g_x) = \mathbb{O}(x, d),
\end{align*}
such that
\begin{enumerate}[label=(\alph*)]
    \item In the \textbf{deterministic} setting, the oracle returns \[ f_x = f(x), \,\, g_x  \in \partial f(x) \text{ satisfying } \iprod{g_x}{d} = f'(x, d).\]
    \item In the \textbf{stochastic finite-variance} setting, the oracle only returns a stochastic gradient $g$ with $\E[g] = g_x$, where $g_x \in \partial f(x)$ satisfies $\iprod{g_x}{d} = f'(x, d)$. Moreover, the variance $\E[\|g-g_x\|^2] \le \sigma^2$ is bounded. In particular, no function value is accessible.
\end{enumerate}
\end{assumption}

We remark that one cannot generally evaluate the generalized gradient $\partial f$ in practice at any point where $f$ is not differentiable. When the function $f$ is not directionally differentiable, one needs to incorporate gradient sampling to estimate $\partial f$~\citep{burke2002approximating}. Our oracle queries only an element of the generalized gradient and is thus \textbf{weaker} than querying the entire set $\partial f$. Still, finding a vector $g_x$ such that $\iprod{g_x}{d}$ equals the directional derivative $f'(x, d)$ is non-trivial in general. Yet, when the objective function is a composition of directionally differentiable functions, such as ReLU neural networks, and if a closed form directional derivative is available for each function in the composition, then we can find the desired $g_x$ by appealing to the chain rule in Lemma~\ref{lemma:chain-rule}. This property justifies our choice of oracles. 

\subsection{Algorithm class and complexity measures}
An algorithm $A$ maps a function $f \in \mathcal{F}(\Delta, L)$ to a sequence of points $\{x_k\}_{k \ge 0}$ in $\mathbb{R}^n$. We denote $A^{(k)}$ to be the mapping from previous $k$ iterations to $x_{k+1}$.  Each $x_k$ can potentially be a random variable, due to the stochastic oracles or algorithm design. Let $\{\mathcal{F}_k\}_{k \ge 0}$ be the filtration generated by $\{x_k\}$ such that $x_k$ is adapted to $\mathcal{F}_k$. Based on the definition of the oracle, we assume that the iterates follow the structure 
\begin{align}
    x_{k+1} = A^{(k)}(x_1, g_1, f_1, x_2, g_2, f_2, ..., x_k, g_k, f_k), 
\end{align}
where $(f_k, g_k) = \mathbb{O}(y_k, d_k)$, and the point $y_k$ and direction $d_k$ are (stochastic) functions of the iterates $x_1,\ldots, x_k$.  

For a random process $\{x_k\}_{k \in \mathbb{N}}$, we define the complexity of $\{x_k\}_{k \in \mathbb{N}}$ for a function $f$ as the value
\begin{equation}
  \begin{split}
    \label{eq:T-sto}
    &T_{\delta, \epsilon}(\{x_t\}_{t \in \mathbb{N}}, f) :=\\
    &\quad \inf \bigl\{t \in \mathbb{N}\  \mid\  \text{Prob}\{d(0, \partial f(x+\delta B)) \ge \epsilon\\
    &\quad  \quad \ \ \text{ for all } k \le t\} \le \tfrac{1}{3}\bigr\}.
  \end{split}
\end{equation}
Let $A[f, x_0]$ denote the sequence of points generated by algorithm $A$ for function $f$. Then, we define the iteration complexity of an algorithm class $\mathcal{A}$ on a function class $\mathcal{F}$ as
\begin{align}\label{eq:complexity}
    \mathcal{N}(\mathcal{A}, \mathcal{F}, \epsilon, \delta) := \inf_{A \in \mathcal{A}} \sup_{\substack{f \in  \mathcal{F}}} T_{\delta, \epsilon}(A[f, x_0], f).
\end{align}
At a high level, \eqref{eq:complexity} is the minimum number of oracle calls required for a fixed algorithm to find a $(\delta, \epsilon)$-stationary point with probability at least $2/3$ for all functions is class $\cal F$.

\vspace*{-5pt}
\section{Deterministic Setting}\label{sec:det}
For optimizing $L$-smooth functions, a crucial inequality is  
\begin{equation}\label{eq:smooth-descent}
    f\bigl(x - \tfrac{1}{L} \nabla f(x) \bigr) -f(x) \leq -\tfrac{1}{2L} \| \nabla f(x) \|^2.
\end{equation}
In other words, either the gradient is small or the function value decreases sufficiently along the negative gradient. However, when the objective function is nonsmooth, this descent property is no longer satisfied. Thus, defining an appropriate descent direction is non-trivial. Our key innovation is to solve this problem via randomization.

More specifically, in our algorithm, Interpolated Normalized Gradient Descent (\ingd), we derive a local search strategy to find the descent direction at an iterate $x_t$. The vector $m_{t,k}$ plays the role of descent direction and we sequentially update it until the condition
\begin{equation}\label{eq:dc}
    f(x_{t,k}) -f(x_t)  <   -\frac{\delta \| m_{t,k} \|}{4}, \tag{descent condition}
\end{equation} 
is satisfied. To connect with the descent property (\ref{eq:smooth-descent}), observe that when $f$ is smooth, with $m_{t,k} = \nabla f(x_t)$ and $\delta = \| m_{t,k}\|/L$, \eqref{eq:dc} is the same as (\ref{eq:smooth-descent}) up to a factor $2$. This connection motivates our choice of descent condition.

When the descent condition is satisfied, the next iterate~$x_{t+1}$ is obtained by taking a normalized step from $x_t$ along the direction $m_{t,k}$. Otherwise, we stay at $x_t$ and continue the search for a descent direction. We raise special attention to the fact that inside the $k$-loop, the iterates $x_{t,k}$ are always obtained by taking a normalized step from $x_t$. Thus, all the inner iterates $x_{t,k}$ have distance exactly $\delta$ from~$x_t$. 

To update the descent direction, we incorporate a randomized strategy. We randomly sample an interpolation point $y_{t,k+1}$ on the segment $[x_t, x_{t,k}]$ and evaluate the generalized gradient $g_{t,k+1}$ at this random point $y_{t,k+1}$. Then, we update the descent direction as a convex combination of $g_{t,k+1}$ and the previous direction $m_{t,k}$. Due to lack of smoothness, the violation of the descent condition does not directly imply that $g_{t,k+1}$ is small. Instead, the projection of the generalized gradient is small along the direction $m_{t,k}$ on average. Hence, with a proper linear combination, the random interpolation allows us to guarantee the decrease of  $\| m_{t,k} \|$ in expectation. 
This reasoning allows us to derive the non-asymptotic convergence rate in high probability.



\begin{algorithm}[tb]
	\caption{ Interpolated Normalized Gradient Descent } \label{algo}
	\begin{algorithmic}[1]
	    \STATE {\bf Initialize $x_1 \in \R^d$}
		\FOR{$t = 1, 2, ..., T$} 
		\WHILE{$\|m_{t,K}\| > \epsilon$} 
		\STATE Call oracle $\sim, m_{t,1} = \mathbb{O}(x_t, \vec{0})$  
		\FOR{$k = 1, ..., K$}
		    \STATE $x_{t,k} = x_t - \delta \frac{m_{t,k}}{\|m_{t,k}\|}$
		  \IF{$\| m_{t,k} \| \le \epsilon$ }
		  \STATE Terminate the algorithm and return $x_t$
		  \ELSIF{ $f(x_{t,k}) -f(x_t)  <   -\frac{\delta \| m_{t,k} \|}{4}$ }
		  \STATE Break while-loop
		  \STATE Set $x_{t+1} = x_{t,k}$ and $t \gets t+1$
		  \ELSE
		  \STATE Sample $y_{t,k+1}$ uniformly from $ [x_t, x_{t,k}]$ 
		  \STATE Call oracle $\sim, g_{t,k+1} = \mathbb{O}(y_{t,k+1}, -m_{t,k})$
          \STATE Update $m_{t,k+1} =  \beta_{t,k} m_{t,k} + (1-\beta_{t,k}) g_{t,k+1} $ with $\beta_{t,k} = \frac{4L^2 - \|m_{t,k} \|^2 }{4L^2+ 2\|m_{t,k}\|^2 }$
		  \ENDIF
		\ENDFOR
		\ENDWHILE
		\ENDFOR
		\STATE Return $x_t$ such that $\|m_{t,K}\| \le \epsilon$
	\end{algorithmic}
\end{algorithm}
\begin{theorem}\label{thm:det}
In the deterministic setting and with Assumption~\ref{assump}(a), the \ingd algorithm 
with parameters $K = \frac{48L^2}{\epsilon^2}$ and $T = \frac{4\Delta}{\epsilon \delta}$ finds a $(\delta, \epsilon)$-stationary point for function class $\mathcal{F}(\Delta, L)$ with probability $1-\gamma$ using at most 
$$\frac{192\Delta L^2}{\epsilon^3\delta}\log\left(\frac{4\Delta}{\gamma\delta\epsilon}\right)\quad \text{oracle calls.}$$
\end{theorem}
Since we introduce random sampling for choosing the interpolation point, even in the deterministic setting we can only guarantee a high probability result. The detailed proof is deferred to  Appendix~\ref{proof:det}. 

A sketch of the proof is as follows. Since $\| x_{t,k} -x_t\| = \delta$ for any $k$, the interpolation point $y_{t,k}$ is inside the ball $x_t+\delta B$. Hence $m_{t,k} \in \partial f(x_t+\delta B)$ for any $k$. In other words, as soon as $ \| m_{t,k} \| \le \epsilon$ (line 7), the reference point $x_t$ is $(\delta,\epsilon)$-stationary. If this is not true, i.e., $ \| m_{t,k} \| > \epsilon$, then we check whether \eqref{eq:dc} holds, in which case
\[  f(x_{t,k}) -f(x_t)  <   -\frac{\delta \| m_{t,k} \|}{4} < -\frac{\epsilon \delta}{4} .\]
Knowing that the function value is lower bounded, this can happen  at most $T=\frac{4\Delta}{\epsilon \delta}$ times. Thus, for at least one $x_t$, the local search inside the while loop is not broken by the descent condition. Finally, given that $\| m_{t,k}\| > \epsilon$ and the descent condition is not satisfied, we show that 
\[ \E[\|m_{t,k+1}\|^2]  \le \left(1 - \frac{\E[\|m_{t,k}\|^2]}{3L^2} \right) \E[\|m_{t,k}\|^2] \] 
This implies that $\E[\| m_{t,k} \|^2]$ follows a decrease of order $O(1/k)$. Hence with $K=O(1/\epsilon^2)$, we are guaranteed to find $\|m_{t,k} \| \le \epsilon$ with high probability. 

\begin{remark}\rm If the problem is smooth, the descent condition is always satisfied in one iteration. Hence the global complexity of our algorithm reduces to $T= O(1/\epsilon \delta)$. Due to the equivalence of the notions of stationarity (Prop.~\ref{prop:equivalence}), with $\delta = O(\epsilon/L)$, our algorithm recovers the standard $O(1/\epsilon^2)$ convergence rate for finding an $\epsilon$-stationary point. In other words, our algorithm can adapt to the smoothness condition.         
\end{remark}



\section{Stochastic Setting}
In the deterministic setting one of the key ingredients used \ingd is to check whether the function value decreases sufficiently. However, evaluating the function value can be computationally expensive, or even infeasible in the stochastic setting. For example, when training neural networks, evaluating the entire loss function requires going through all the data, which is impractical. As a result, we do not assume access to function value in the stochastic setting and instead propose a variant of \ingd that only relies on gradient information.   

\begin{algorithm}[htbp]
	\caption{Stochastic \ingd($x_{1}, p, q, \beta, T, K$)}\label{algo:sto}
	\begin{algorithmic}[1]
	    \STATE {\bf Initialize} $x_1 \in \R^d$.
	    \STATE Call oracle $ g(x_1) = \mathbb{O}(x_{1}, \vec{0})$ and set $m_1=g(x_1)$.
		\FOR{$t = 1, 2, ..., T$}
		\STATE Update $x_{t+1} = x_{t} - \eta_{t} m_{t} $ with $\eta_{t}=\tfrac{1}{p\|m_{t}\|+ q}$.
		\STATE Sample $y_{t+1}$ uniformly from $ [x_{t}, x_{t+1}]$
		\STATE Call oracle $g(y_{t+1}) = \mathbb{O}(y_{t+1}, - m_{t})$
		\STATE Update $m_{t+1} = \beta m_{t} + (1-\beta) g(y_{t+1}) $
		
		\ENDFOR
		\STATE Randomly sample $i$ uniformly from $\{1, ..., T\}$.
		\STATE Update $i = \max\{i-K, 1\}$
		\STATE Return $x_i$.
	\end{algorithmic}
\end{algorithm}
One of the challenges of using stochastic gradients is the noisiness of the gradient evaluation. To control the variance of the associated updates, we introduce a parameter $q$ into the normalized step size:
\[ \eta_{t} = \frac{1}{p\|m_{t}\|+ q}.\]
A similar strategy is used in adaptive methods like \cite{adagrad,adam} to prevent instability. Here, we show that the constant $q$ allows us to control the variance of $x_{t+1} - x_{t}$. In particular, it implies the bound
\[ \E [ \|x_{t+1} - x_{t} \|^2 ] \le \frac{G^2}{q},\]
where $G^2:= L^2 + \sigma^2$ is a trivial upper-bound on the expected norm of any sampled gradient $g$. 

Another substantial change (relative to \ingd) is the removal of the explicit local search, since the stopping criterion can now no longer be tested without access to the function value. Instead, one may view $x_{t-K+1}, \dotsc , x_{t-1}, x_{t}$ as an implicit local search with respect to the reference point $x_{t-K}$. In particular, we show that when the direction $m_t$ has a small norm, then $x_{t-K}$ is a $(\delta,\epsilon)$-stationary point, but not~$x_t$. This discrepancy explains why we output $x_{t-K}$ instead of $x_t$.

In the deterministic setting, the direction $m_{t,k}$ inside each local search is guaranteed to belong to $\partial f(x_{t}+\delta B)$. Hence, controlling the norm of $m_{t,k}$ implies the $(\delta,\epsilon)$-stationarity of $x_t$.
In the stochastic case, however, we have two complications. First, only the expectation of the gradient evaluation satisfies the membership $\E[g(y_k)] \in \partial f(y_k)$. Second, the direction $m_t$ is a convex combination of all the previous gradients $g(y_1), \dotsc , g(y_t)$, with all coefficients being nonzero. In contrast, we use a re-initialization in the deterministic setting. We overcome these difficulties and their ensuing subtleties to finally obtain the following complexity result:

\begin{theorem}\label{thm:sto}
In the stochastic setting, with Assumption~\ref{assump}(b), the Stochastic-\ingd algorithm (Algorithm~\ref{algo:sto}) with parameters $G= \sqrt{L^2 + \sigma^2}$, $ \beta = 1 - \frac{\epsilon^2}{64 G^2}$, $p  = \frac{ 64 G^2 \ln(16G/\epsilon)}{\delta \epsilon^2 }$, $q  = 4Gp $, $K=p\delta$, $T = \frac{2^{16} G^3 \Delta \,\, \text{ln}(16G/\epsilon)}{ \epsilon^4 \delta}\max \{1, \frac{  G \delta}{8\Delta} \}$ ensures 
\[ \frac{1}{T} \sum_{t=1}^T \E[\| m_{t} \| ] \le \frac{\epsilon}{4}. \]
In other words, the number of gradient calls to achieve a $(\delta, \epsilon)-$stationary point is upper bounded by 
$  \tilde{\cO} \left (\frac{G^3 \Delta}{\epsilon^4 \delta} \right ). $
\end{theorem}

For readability, the constants in Theorem~\ref{thm:sto} have not been optimized. The high level idea of the proof is to relate $\E[ \eta_t\| m_t\|^2]$ to the function value decrease $f(x_t)-f(x_{t+1})$, and then to perform a telescopic sum. 

We would like to emphasize the use of the adaptive step size $\eta_t$ and the momentum term $m_{t+1}$. These techniques arise naturally from our goal to find a $(\delta, \epsilon)$-stationary point. The step size $\eta_t$ helps us ensure that the distance moved is at most $\frac{1}{p}$, and hence we are certain that adjacent iterates are close to each other. The momentum term $m_t$ serves as a convex combination of generalized gradients, as postulated by Definition~\ref{def:stationary}.

Further, even though the parameter $K$ does not directly influence the updates of our algorithm, it plays an important role in understanding our algorithm. Indeed, we show that
\[ d\left (\E[m_{t} | x_{t-K}],\partial f(x_{t-K}+\delta B) \right ) \le \frac{\epsilon}{16}. \]
In other words, the conditional expectation $\E[m_{t} | x_{t-K}]$ is approximately in the $\delta$-subdifferential $\partial f(x_{t-K}+\delta B)$ at $x_{t-K}$.  This relationship is non-trivial.

On one hand, by imposing $ K \le \delta p$, we ensure that $x_{t-K+1}, \dotsc, x_{t}$ are inside the $\delta$-ball of center $x_{t-K}$. On the other hand, we guarantee that the contribution of $m_{t-K}$ to $m_t$ is small, providing an appropriate upper bound on the coefficient~$\beta^K$. These two requirements help balance the different parameters in our final choice. Details of the proof may be found in Appendix~\ref{proof:sto}. 

Recall that we do not access the function value in this stochastic setting, which is a strength of the algorithm. In fact, 
%
we can show that our ${\delta^{-1}}$ dependence is tight, when the oracle has only access to generalized gradients.

\begin{theorem}[Lower bound on $\delta$ dependence]\label{thm:lower}
Let $\mathcal{A}$ denote the class of algorithms defined in Section~\ref{sec:alg-class} and $\mathcal{F}(\Delta, L)$ denote the class of functions defined in Equation \eqref{eq:fclass}. Assume $\epsilon \in (0,1)$ and $L = 1$.  Then the iteration complexity is lower bounded by $\frac{\Delta}{8\delta}$ if the algorithm \textbf{only} has access to generalized gradients.
\end{theorem}
The proof is inspired by Theorem 1.1.2 in \cite{nesterov2018lectures}. We show that unless more than $\frac{\Delta}{8\delta}$ different points are queried, we can construct two different functions in the function class that have gradient norm $1$ at all the queried points, and the stationary points of both functions are $\Omega(\delta)$ away. For more details, see Appendix~\ref{app:lower}. 

This theorem also implies the negative result for finite time analyses that we showed in Theorem~\ref{thm:finite}. Indeed, when an algorithm finds an $\epsilon$-stationary point, the point is also a  $(\delta,\epsilon)$-stationary for any $\delta>0$.  Thus, the iteration complexity must be at least $\lim_{\delta \rightarrow 0 } \frac{\Delta}{8\delta} = +\infty$, i.e., no finite time algorithm can guarantee to find an $\epsilon$-stationary point. 


Before moving on to the experimental section, we would like to make several comments related to different settings. First, since the stochastic setting is strictly stronger than the deterministic setting, the stochastic variant Stochastic-INGD is applicable to the deterministic setting too. Moreover, the analysis can be extended to $q=0$, which leads to a complexity of $\cO(1/\epsilon^3 \delta)$. This is the same as the deterministic algorithm. However, the stochastic variant does not adapt to the smoothness condition. In other words, even if the function is differentiable, we will not obtain a faster convergence rate. In particular, if the function is smooth, by using the equivalence of the types of stationary points, Stochastic-INGD finds an $\epsilon$-stationary point in $\cO(1/\epsilon^5)$ while standard SGD enjoys a $\cO(1/\epsilon^4)$ convergence rate. We do not know whether a better convergence result is achievable, as our lower bound does not provide an explicit dependency on $\epsilon$; we leave this as a future research direction.





\section{Experiments}
\begin{figure}[ht]
\centering
\begin{subfigure}
	\centering
	\includegraphics[width=0.8\linewidth]{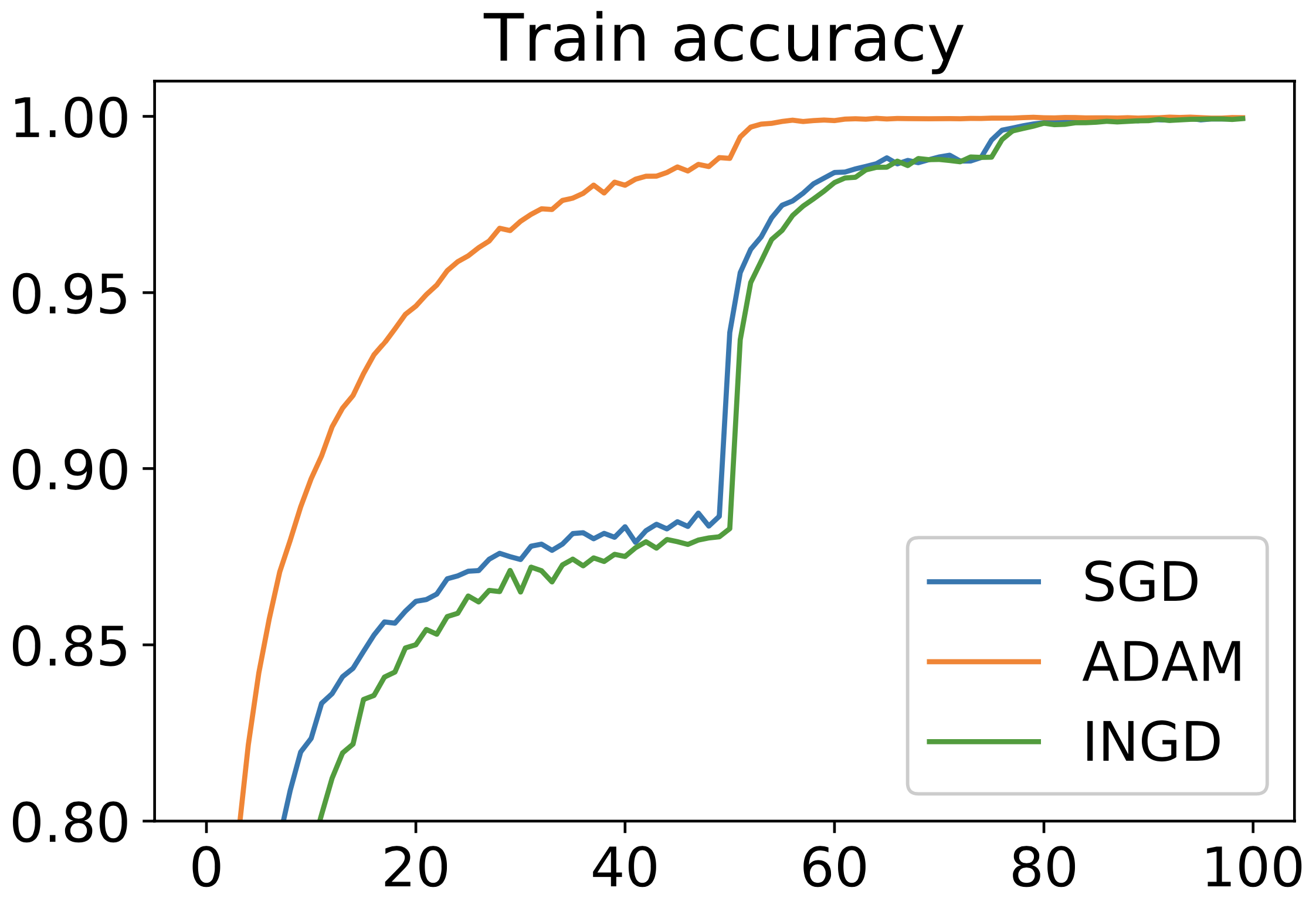}
\end{subfigure}
\begin{subfigure}
	\centering
	\includegraphics[width=0.8\linewidth]{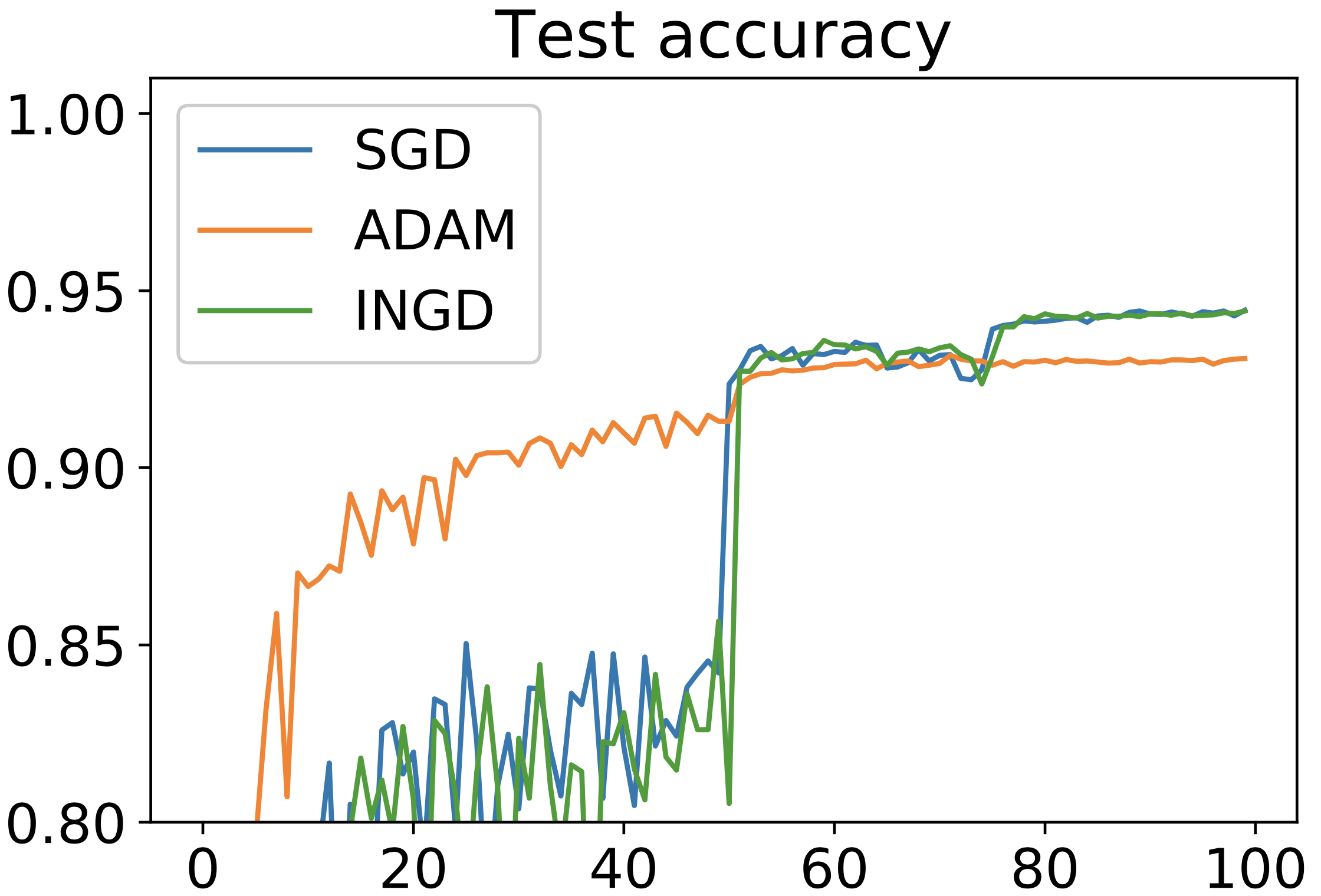}
\end{subfigure}
	\caption{Learning curve of SGD, ADAM and \ingd on training ResNet 20 on CIFAR10.}\label{fig}
\end{figure}
In this section, we evaluate the performance of our proposed algorithm Stochastic~\ingd on image classification tasks.

We train the ResNet20~\citep{he2016deep} model on the CIFAR10~\citep{krizhevsky2009learning} classification dataset. The dataset contains 50k training images and 10k test images in 10 classes. 

We implement Stochastic~\ingd in PyTorch with the inbuilt auto differentiation algorithm \cite{paszke2017automatic}. We remark that except on the kink points, the auto differentiation matches the generalized gradient oracle, which justifies our choice. We benchmark the experiments with two popular machine learning optimizers, SGD with momentum and ADAM~\cite{adam}. We train the model for 100 epochs with the standard hyper-parameters from the Github repository\footnote{https://github.com/kuangliu/pytorch-cifar}: 
\begin{itemize}[leftmargin=1em]
    \item For SGD with momentum, we initialize the learning rate as $0.1$, momentum as $0.9$ and reduce the learning rate by 10 at epoch 50 and 75. The weight decay parameter is set to $5 \cdot 10^{-4}$.
    \item For ADAM, we use constant the learning rate $10^{-3}$, betas in $(0.9, 0.999)$, and weight decay parameter $10^{-6}$ and $\epsilon = 10^{-3}$ for the best performance. 
    \item For Stochastic-\ingd, we use $\beta = 0.9$, $p=1$, $q=10$, and weight decay parameter $5 \times 10^{-4}$.
\end{itemize}
The training and test accuracy for all three algorithms are plotted in Figure~\ref{fig}. We observe that Stochastic-\ingd matches the SGD baseline and outperforms the ADAM algorithm in terms of test accuracy. The above results suggests that the experimental implications of our algorithm could be interesting, but we leave a more systematic study as future direction.


\section{Conclusions and Future Directions}
In this paper, we investigate the complexity of finding first order stationary points of nonconvex nondifferentiable functions. We focus in particular on Hadamard semi-differentiable functions, which we suspect is perhaps the most general class of functions for which the chain rule of calculus holds---see the monograph~\citep{delfour2019introduction}. We further extend the standard definition of $\epsilon$-stationary points for smooth functions into a new notion of $(\delta, \epsilon)$-stationary points. We justify our definition by showing that no algorithm can find a $(0, \epsilon)$ stationary point for any $\epsilon < 1$ in a finite number of iterations and conclude that a positive $\delta$ is necessary for a finite time analysis. Using the above definition and a more refined gradient oracle, we prove that the proposed algorithms find stationary points within $\mathcal{O}(\epsilon^{-3}\delta^{-1})$ iterations in the deterministic setting and with $\mathcal{O}(\epsilon^{-4}\delta^{-1})$ iterations in the stochastic setting. 

Our results provide the first non-asymptotic analysis of nonconvex optimization algorithms in the general Lipschitz continuous setting. Yet, they also open further questions. The first question is whether the current  dependence on $\epsilon$  in our complexity bound is optimal. A future research direction is to try to find provably faster algorithms or construct adversarial examples that close the gap between upper and lower bounds on $\epsilon$. Second, the rate we obtain in the deterministic case requires function evaluations and is randomized, leading to high probability bounds. Can similar rates be obtained by an algorithm oblivious to the function value? Another possible direction would be to obtain a deterministic convergence result. More specialized questions include whether one can remove the logarithmic factors from our bounds. Aside from the above questions on the rate, we can take a step back and ask high-level questions. Are there better alternatives to the current definition of $(\delta,\epsilon)$-stationary points? One should also investigate whether everywhere directional differentiability is  necessary.

In addition to the open problems listed above, our work uncovers another very interesting observation. In the standard stochastic, nonconvex, and smooth setting, stochastic gradient descent is known to be theoretically optimal~\citep{arjevani2019lower}, while widely used practical techniques such as momentum-based and adaptive step size methods usually lead to worse theoretical convergence rates. In our proposed setting, momentum and adaptivity naturally show up in algorithm design, and become necessary for the convergence analysis. Hence we believe that studying optimization under more relaxed assumptions may lead to theorems that can better bridge the widening theory-practice divide in optimization for training deep neural networks, and ultimately lead to better insights for practitioners.


\section{Acknowledgement}
The authors thank Ohad Shamir for helpful discussions, and for pointing out the difference between being $(\delta, \epsilon)$ stationary and being $\delta$ close to an $\epsilon$ stationary point. This work was partially supported by the MIT-IBM Watson AI Lab. SS and JZ also acknowledge support from NSF CAREER grant Number 1846088.

\bibliography{ref}
\bibliographystyle{icml2020}

\onecolumn
\appendix


\section{Proof of Lemmas in Preliminaries}
\subsection{Proof of Lemma~\ref{lemma:lebesgue-integral}}\label{proof:lebesgue-integral}
\begin{proof}
Let $g(t) = f(x+t(y-x))$ for $t\in [0,1]$, then $g$ is $L\| y-x \|$-Lipschitz implying that $g$ is absolutely continuous. Thus from the fundamental theorem of calculus (Lebesgue), $g$ has a derivative $g'$ almost everywhere, and the derivative is Lebesgue integrable such that 
\[g(t) = g(0)+\int_0^t g'(s) ds. \]
Moreover, if $g$ is differentiable at $t$, then 
\[ g'(t) = \lim_{\delta t \rightarrow 0} \frac{g(t+\delta t)-g(t)}{\delta t} =  \lim_{\delta t \rightarrow 0} \frac{f(x+ (t+\delta t)(y-x))-f(x+t(y-x))}{\delta t}  = f'(x+t(y-x), y-x). \]
Since this equality holds almost everywhere, we have  
\[f(y)-f(x) = g(1)-g(0) =\int_0^1 g'(t) dt = \int_0^1 f'(x+t(y-x), y-x) dt.  \]
\end{proof}

\subsection{Proof of Lemma~\ref{lemma:directional-generalized}}\label{proof:directional-generalized}
\begin{proof}
For any $\varphi(t) = x + td$ as given in Definition~\ref{def:direction}, let $t_k \to 0$. Denote $x_k = \varphi(t_k), \delta_k = \|x_k - x\| \to 0$. By Proposition 1.6, we know that there exists $g_{k, j} \in \cup_{y \in x + \delta_k B} \partial f(y)$ such that 
$$f(x_k) - f(x) = \iprod{g_{k, j}}{x_k - x}.$$
By the existence of directional derivative, we know that 
$$\lim_{k \to \infty} \iprod{g_{k, j}}{ d} = \lim_{k \to \infty} \frac{\iprod{g_{k, j}}{t_k d}}{t_k}= f'(x, d) $$
$g_{k, j}$ is in a bounded set with norm less than L. The Lemma follows by the fact that any accumulation point of $g_{k, j}$ is in $\partial f(x)$ due to upper-semicontinuity of $\partial f(x)$.
\end{proof}

\section{Proof of Lemmas in Algorithm Complexity}

\subsection{Proof of Theorem~\ref{thm:finite}}\label{proof:finite}

Our proof strategy is similar to Theorem 1.1.2 in \cite{nesterov2018lectures}, where we use the resisting strategy to prove lower bound. Given a one dimensional function $f$, let $x_{k}, k \in [1,K]$ be the sequence of points queried in ascending order instead of query order. We assume without loss of generality that the initial point is queried and is an element of $\{x_k\}_{k=0}^K$ (otherwise, query the initial point first before proceeding with the algorithm). 

Then we define the resisting strategy: always return
\begin{align*}
    f(x) = 0,\text{and} \quad \nabla f(x) = L.
\end{align*}

If we can prove that for any set of points $x_{k}, k \in [1,K]$, there exists two functions such that they satisfy the resisting strategy $f(x_k) = 0,\text{and} \quad \nabla f(x_k) = L, k \in [1,K]$, and that the two functions do not share any common stationary points, then we know no randomized/deterministic can return an $\epsilon-$stationary points with probability more than $1/2$ for both functions simultaneously. In other word, no algorithm that query $K$ points can distinguish these two functions. Hence we proved the theorem following the definition of complexity in~\eqref{eq:complexity} with $\delta = 0$.

All we need to do is to show that such two functions exist in the Lemma below.

\begin{lemma}
Given a finite sequence of real numbers $\{x_{k}\}_{k\in[1,K]} \in \R $, there is a family of functions $f_{\theta} \in \mathcal{F}(\Delta, L)$ such that for any $k \in [1,K]$, 
\[  f_{\theta}(x_k) = 0 \quad \text{and} \quad \nabla f_{\theta}(x_k) = L\]
and for $\epsilon$ sufficiently small, the set of $\epsilon$-stationary points of $f_{\theta}$ are all disjoint, i.e
$\{ \epsilon$-stationary points of $f_{\theta_1}\} \cap \{ \epsilon$-stationary points of $f_{\theta_2} \}= \emptyset$ for any $\theta_1 \neq \theta_2$. 
\end{lemma}
\begin{proof}
Up to a permutation of the indices, we could reorder the sequence in the increasing order. WLOG, we assume $x_k$ is increasing. 
Let $\delta = \min\{ \min_{x_i \neq x_j} \{ |x_i -x_j|\}, \frac{\Delta}{L} \}$.
For any $0<\theta<1/2$, we define $f_\theta$ by 
\begin{align*}
    f_\theta(x) & = -L(x-x_{1}+2\theta \delta)  \quad \text{for} \quad x \in(-\infty, x_1-\theta \delta] \\
    f_\theta(x) & = L(x-x_{k}) \quad \text{for} \quad x \in \left [x_k -\theta \delta, \frac{x_k+x_{k+1}}{2}-\theta \delta \right ] \\
     f_\theta(x) & = -L(x-x_{k+1}+2\theta \delta)  \quad \text{for} \quad x \in \left [\frac{x_k+x_{k+1}}{2}-\theta \delta, x_{k+1} -\theta \delta \right ]\\ 
    f_\theta(x) &= L(x-x_K) \quad x \in [x_K+ \theta \delta, +\infty).
\end{align*} 
It is clear that $f_\theta$ is directional differentiable at all point and $\nabla f_\theta(x_k) = L $. Moreover, the minimum $f_\theta^* = -L\theta \delta \ge -\Delta$. This implies that $f_{\theta} \in \mathcal{F}(\Delta, L)$. Note that $\nabla f_\theta =L$ or $-L$ except at the local extremum. Therefore, for any $\epsilon < L$ the set of $\epsilon$-stationary points of $f_\theta$ are exactly 
\[ \{ \epsilon\text{-stationary points of }f_{\theta} \} = \{ x_{k} - \theta \delta \,\, | \,\, k \in [1,K]\} \cup \left \{ \frac{x_k+x_{k+1}}{2} - \theta \delta \,\, | \,\, k \in [1,K-1] \right \}, \]
which is clearly distinct for different choice of $\theta$.
\end{proof}

\subsection{Proof of Proposition~\ref{prop:equivalence}\label{proof:prop-equiv}}
\begin{proof}
When $x$ is $(\frac{\epsilon}{3L}, \frac{\epsilon}{3})$ stationary, we have $d(0, \partial f(x+ \frac{\epsilon}{3L}B)) \le \frac{\epsilon}{3}$. By definition, we could find $g \in conv( \cup_{y \in x+ \frac{\epsilon}{3L} B} \nabla f(y) )$ such that $\| g \| \le 2\epsilon/3$. This means, there exists $x_1, \cdots,  x_k \in x+ \frac{\epsilon}{3L} B$, and $\alpha_1, \cdots, \alpha_k \in [0,1]$ such that $\alpha_1 + \cdots + \alpha_k =1$ and 
\[ g = \sum_{i=1}^k \alpha_i \nabla f(x_i) \]
Therefore 
\begin{align*}
    \| \nabla f(x) \| & \le \| g\|+ \| \nabla f(x) -g \|  \\
    & \le \frac{2\epsilon}{3}+ \sum_{i=1}^k \alpha_i \| \nabla f(x) - \nabla f(x_k) \|  \\
    & \le \frac{2\epsilon}{3}+ \sum_{i=1}^k \alpha_i L \|x - x_k \| \\
    & \le \frac{2\epsilon}{3}+ \sum_{i=1}^k \alpha_i L \frac{\epsilon}{3L} = \epsilon.
\end{align*}
Therefore, $x$ is an $\epsilon$-stationary point in the standard sense. 
\end{proof}

\subsection{Proof of Lemma~\ref{lemma:approx}}\label{proof:lemma-approx}
\begin{proof}
    First, we show that the limit exists. By Lipschitzness and Jenson inequality, we know that $\partial f(x + \delta_{k+1}B)$ lies in a bounded ball with radius $L$. For any sequence of $\{\delta_k\}$ with $\delta_k \downarrow 0$,  we know that $\partial f(x + \delta_{k+1}B) \subseteq \partial f(x + \delta_{k}B).$ Therefore, the limit exists by the monotone convergence theorem. \\

    Next, we show that $\lim_{\delta \downarrow 0} \partial f(x+\delta B) = \partial f(x).$ For one direction, we show that $\partial f(x) \subseteq  \lim_{\delta \downarrow 0} \partial f(x+\delta B) $. This follows by proposition 1.5 and the fact that 
    \begin{align*}
    \cup_{y \in x + \delta B}\partial f(y) \subseteq \text{conv}(\cup_{y \in x + \delta B}\partial f(y))  =  \partial f(x+\delta B).
    \end{align*}

    Next, we show the other direction $  \lim_{\delta \downarrow 0} \partial f(x+\delta B) \subseteq \partial f(x)$. By upper semicontinuity, we know that
    for any $\epsilon > 0$, there exists $\delta > 0$ such that
    \begin{align*}
     \cup_{y \in x + \delta B}\partial f(y) \subseteq \partial f(x) + \epsilon B.
    \end{align*}
    Then by convexity of $\partial f(x)$ and $\epsilon B$, we know that their Minkowski sum $\partial f(x) + \epsilon B$ is convex. Therefore, we conclude that for any $\epsilon > 0$, there exists $\delta > 0$ such that
    \begin{align*}
     \partial f(x+\delta B) = \text{conv}(\cup_{y \in x + \delta B}\partial f(y)) \subseteq \partial f(x) + \epsilon B.
    \end{align*}

\end{proof}

\section{Proof of Theorem~\ref{thm:det}}\label{proof:det}


Before we prove the theorem, we first analyze how many times the algorithm iterates in the while loop.

\begin{lemma} \label{lemma:det-while}
Let $K = \frac{48L^2}{\epsilon^2}$. Given $t \in [1,T]$, 
\[\E[\|m_{t,K}\|^2] \le \frac{\epsilon^2}{16}. \]
where for convenience of analysis, we define $m_{t,k} = 0$ for all $k>k_0$ if the $k$-loop breaks at $(t,k_0)$. Consequently, for any $\gamma<1$, with probability $1 - \gamma$, there are at most $log(1/\gamma)$ restarts of the while loop at  the $t$-th iteration. 
\end{lemma}

\begin{proof}
Let $\fF_{t,k}= \sigma(y_{t,1}, \cdots, y_{t,k+1})$, then $x_{t,k}, m_{t,k} \in \fF_{t,k}.$ We denote $D_{t,k}$ as the event that $k$-loop does not break at $x_{t,k}$, i.e. $\|m_{t,k}\| > \epsilon$ and $f(x_{t,k}) -f(x_t) >  -\frac{\delta \| m_{t,k} \|}{4}$. It is clear that $D_{t,k} \in \fF_{t,k}$. 

Let $\gamma(\lambda) = (1-\lambda) x_{t} + \lambda x_{t, k}, \ \lambda \in [0,1]$. Note that $\gamma'(\lambda) =  x_{ t,k} - x_{ t} = - \delta \frac{m_{t,k}}{\|m_{t,k}\|}$.
Since $y_{t,k+1}$ is uniformly sampled from line segment $[x_t, x_{t,k}]$, we know 
$$\E[\iprod{g_{t,k+1}}{x_{t,k} - x_t} | \fF_{t,k}] = \int_0^1 f'(\gamma(t), x_{t,k} - x_t) dt = f(x_{t,k}) - f(x_t) $$
where the second equality comes from directional differentiability.
Since $x_{ k+1} - x_{ k} = - \delta \frac{m_{t,k}}{\|m_{t,k}\|}$, we know that
\begin{align}
    \E[\iprod{g_{t,k+1}}{m_{t,k}} | \fF_{t,k} ] = -\frac{ \|m_{t,k}\|}{\delta} (f(x_{t,k}) - f(x_t)).
\end{align}
By construction  $m_{t,k+1} = \beta m_{t,k} + (1-\beta) g_{t,k+1}$ under $D_{t,k} \cap \cdots \cap D_{t,1}$, and  $m_{t,k+1}=0$ otherwise.
Therefore,
\begin{align*}
    & \E[\|m_{t,k+1}\|^2 | \fF_{t,k} ]  \\
     = &  \E [\|\beta m_{t,k} + (1-\beta)g_{t,k+1}\|^2 \mathds{1}_{D_{t,k} \cap \cdots \cap D_{t,1}} | \fF_{t,k} ] \\
     \le&  \left (\beta^2 \|m_{t,k}\|^2 + (1-\beta)^2 L^2 + 2\beta(1-\beta)\E[\iprod{g_{t,k+1}}{m_{t,k}} | \fF_{t,k} ] \right ) \mathds{1}_{D_{t,k} \cap \cdots \cap D_{t,1}}  \\
    \le & \beta^2 \|m_{t,k}\|^2 + (1-\beta)^2 L^2 - 2\beta(1-\beta)\frac{\|m_{t,k}\|}{\delta} ( f(x_{t,k})-f(x_t)) \mathds{1}_{D_{t,k} \cap \cdots \cap D_{t,1}} \\ 
    \le &  \beta^2 \|m_{t,k}\|^2 + (1-\beta)^2 L^2 + 2\beta(1-\beta)\frac{\|m_{t,k}\|^2}{4} \\
\end{align*}
where in the third line, we use the fact $\beta,  D_{t,k} \cap \cdots \cap D_{t,1} \in \fF_{t,k}$; in the fourth line we use the fact under $D_{t,k}$, $f(x_{t,k})-f(x_t) \ge - \frac{\delta \| m_{t,k}\|}{4}$.
The last equation is a quadratic function with respect to $\beta$, which could be rewritten as 
\[h(\beta) = \beta^2(\frac{\|m_{t,k}\|^2}{2} + L^2) -2 \beta (L^2 - \frac{\|m_{t,k}\|^2}{4}) + L^2.\] 
It achieves the minimum at $\beta = \frac{4L^2 - \|m_{t,k}\|^2}{ 4L^2+ 2\|m_{t,k}\|^2}$, which belongs to $\fF_{t,k}$. Since $\|m_{t,k}\| \le L $, we have
\[ h^* = \frac{L^2}{L^2+ \frac{\|m_{t,k}\|^2}{2}} \|m_{t,k}\|^2  \le  \left(1 - \frac{\|m_{t,k}\|^2}{3L^2} \right) \|m_{t,k}\|^2\]
Therefore,
\begin{align*}
    & \E [\|m_{t,k+1}\|^2 ] \\
 =  & \E [ \E[\|m_{t,k+1}\|^2 | \fF_{t,k}] ] \\
 \le & \E \left[ \left(1 - \frac{\|m_{t,k}\|^2}{3L^2} \right) \|m_{t,k}\|^2 \right ] \\
 \le & \left(1 - \frac{\E[\|m_{t,k}\|^2]}{3L^2} \right) \E[\|m_{t,k}\|^2]
\end{align*}

where the last inequality follows from Jensen's inequality under the fact that the function $x \rightarrow (1-x/3L^2)x$ is concave.
Now consider the sequence $v_k = \E[\|m_{t,k}\|^2]/L^2 \in [0,1]$, we get
$$v_{k+1} \le  v_k - v_k^2/3 \quad \implies \quad \frac{1}{v_{k+1}} \ge \frac{1}{v_k - v_k^2/3} \ge \frac{1}{v_k} + \frac{1}{3}.$$ 
Knowing that $v_1 \le 1$, we therefore have 
$$v_{k} \le \frac{3}{k + 2}.$$

When $K > \frac{48 L^2}{\epsilon^2}$, we have
$\E[\|m_{t,K}\|^2 ] \le \frac{\epsilon^2}{16}.$
Therefore, by Markov inequality, $\mathcal{P}\{\|m_{t,K}\| \ge \epsilon \} \le 1/4$. In other word, the while-loop restart with probability at most $1/4$.  
Therefore, with probability $1-\gamma$, there are at most $\log(1/\gamma)$ restarts.
\end{proof}

Now we are ready to prove the main theorem.

\begin{proof}[Proof of Theorem \ref{thm:det}]
We notice that $m_{t,k}$ is always a convex combinations of generalized gradients within the $\delta$ ball of $x_k$, i.e.
\[ m_{t,k} \in     \partial f(x_t+\delta B) = \text{conv}(\cup_{y \in x_t + \delta B}\partial f(y)) \]
Therefore, if at any $t,k$, $\|m_{t,k}\| \le \epsilon$, then the corresponding $x_t$ is a $(\delta, \epsilon)$ approximate stationary point. To show that our algorithm always find a $\|m_{t,k}\| \le \epsilon$, we need to control the number of times the descent condition is satisfied, which breaks the while-loop without satisfying $\|m_{t,k}\| \le \epsilon$. Indeed, when the descent condition holds, we have
\[ f(x_{t,k})- f(x_t) \le -\frac{\delta \| m_{t,k} \|}{4} < -\frac{\delta \epsilon}{4},  \]
where we use the fact $\| m_{t,k} \| > \epsilon$, otherwise, the algorithm already terminates. 
Consequently, there are at most $\frac{4\Delta}{\delta\epsilon}-1= T-1$ iterations that the descent condition holds. As a result, for at least one $t$ , the while-loop ends providing a $(\delta, \epsilon)$ approximate stationary point.

By Lemma~\ref{lemma:det-while}, we know that with probability $1-\frac{\gamma \delta\epsilon}{4\Delta}$, the $t$-th iteration terminates in $\log(\frac{4\Delta}{\gamma\delta\epsilon})$ restarts. 
Consequently, with probability $1-\gamma$, the algorithm returns a $(\delta, \epsilon)$ approximate stationary point using
$$ \frac{192\Delta L^2}{\epsilon^3\delta}\log\left(\frac{4\Delta}{\gamma\delta\epsilon}\right)  \ \ \text{oracle calls.}$$

\end{proof}

\section{Proof of Theorem~\ref{thm:sto}}\label{proof:sto}

		

Stochastic INGD has convergence guarantee as stated in the next theorem.

\begin{theorem}

Under the stochastic Assumption~\ref{assump}, the Stochastic INGD algorithm in Algorithm~\ref{algo:sto} with parameters $ \beta = 1 - \frac{\epsilon^2}{64 G^2}$, $p  = \frac{ 64 G^2 \ln(16G/\epsilon)}{\delta \epsilon^2 }$, $q  = 4Gp $, $T = \frac{2^{16} G^3 \Delta \,\, \text{ln}(16G/\epsilon)}{ \epsilon^4 \delta}\max \{1, \frac{  G \delta}{8\Delta} \}$, $K = p\delta$ has algorithm complexity upper bounded by
$$ \frac{2^{16} G^3 \Delta \,\, \text{ln}(16G/\epsilon)}{ \epsilon^4 \delta}\max \{1, \frac{  G \delta}{8\Delta} \} = \tilde{O} \left (\frac{G^3 \Delta}{\epsilon^4 \delta} \right ). $$

\end{theorem}

\begin{proof} 
First, we are going to show that  
\begin{equation}\label{eq: m-avg}
    \frac{1}{T} \sum_{t=1}^T \E [ \| m_{t}\|] \le \epsilon/4.
\end{equation}


From construction of the descent direction, we have
\begin{align}\label{eq:sto-proof-m}
   \|m_{t+1}\|^2 
   = (1-\beta)^2\|g(y_{t+1})\|^2 + 2\beta(1-\beta)\iprod{g(y_{t+1})}{m_{t}} + 
   \beta^2 \|m_{t}\|^2 .
\end{align}
Multiply both side by $\eta_{t}$ and sum over $t$, we get
\begin{align}\label{eq:=0}
   0 = (1-\beta)^2 \underbrace{\sum_{t=1}^T \eta_{t} \|g(y_{t+1})\|^2}_{\text{i}} + 2\beta(1 - \beta) \underbrace{\sum_{t=1}^T  \iprod{g(y_{t+1})}{\eta_{t} m_{t}}}_{\text{ii}}  +&  \underbrace{\sum_{t=1}^{T}  \eta_{t}(-\|m_{t+1}\|^2  + \beta^2 \|m_{t}\|^2 )}_{\text{iii}}.
\end{align}


We remark that at each iteration, we have two randomized/stochastic procedure: first we draw $y_{t+1}$ randomly between the segment $[x_t,x_{t+1}]$, second we draw a stochastic gradient at $y_{t+1}$.
For convenience of analysis, we denote $\mathcal{G}_{t}$ as the sigma field generated by $g(y_t)$, and $\mathcal{Y}_{t}$ as the sigma field generated by $y_{t}$. Clearly, $\mathcal{G}_{t} \subset \mathcal{Y}_{t+1} \subset \mathcal{G}_{t+1}$. By definition $\eta_{t}$ is determined by $m_{t}$, which is further determined by $g_t$. Hence, the vectors $m_t$, $\eta_t$ and $x_{t+1}$ are $\mathcal{G}_{t}$-measurable. 

Now we analyze each term one by one.

\textbf{Term i:} This term could be easily bound by
\begin{align}
    \E[\eta_{t} \|g(y_{t+1})\|^2] \le \frac{1}{q}\E[ \|g(y_{t+1})\|^2] =  \frac{1}{q} \E[\E[ \|g(y_{t+1})\|^2| \Y_{t+1}]] \le \frac{G^2}{ q}
\end{align}

\textbf{Term ii:}
 Note that $\eta_t m_t = x_t -x_{t+1}$, we have
\begin{align*}
   \E[\iprod{g(y_{t+1})}{\eta_{t} m_{t}}\ |\ \mathcal{G}_{t}] &= \E[\E[\iprod{g(y_{t+1})}{x_t-x_{t+1}}\ | \ \mathcal{Y}_{t+1} ]|\ \mathcal{G}_{t}] \\
   &= \E[ f'(y_{t+1}; x_{t}-x_{t+1})| \ \mathcal{G}_{t}] \\
   &= \int_{[0,1]} f'(x_{t+1}+ \lambda(x_{t}-x_{t+1}); x_{t}-x_{t+1}) d\lambda \\
   & = f(x_{t}) - f(x_{t+1}),
\end{align*}
where the second line we use the property of the oracle given in Assumption 1(b). Thus by taking the expectation, we have 
\[ \sum_{t=1}^T  \E[\iprod{g(y_{t+1})}{\eta_{t}m_t}] = \E[ f(x_1) -f(x_{T+1})] \le \Delta \]


\textbf{Term iii:} we would like to develop a telescopic sum for the third term, however this is non-trivial since the stepsize $\eta_t$ is adaptive. Extensive algebraic manipulation is involved.
\begin{align}\label{eq:term2}
     & \sum_{t=1}^{T} \eta_{t}(-\|m_{t+1}\|^2  + \beta^2 \|m_{t}\|^2 ) \nonumber \\
     = & \sum_{t=1}^{T} \frac{-\|m_{t+1}\|^2}{p\|m_{t}\| + q} + \beta^2  \sum_{t=1}^{T} \frac{ \|m_{t}\|^2}{p\|m_{t}\| + q} \nonumber \\
     = & \sum_{t=1}^T \left ( \frac{-\|m_{t+1}\|^2}{p\|m_{t}\| + q} + \frac{\|m_{t+1}\|^2}{p\|m_{t+1}\| + q}  \right ) -\sum_{t=1}^T \frac{\|m_{t+1}\|^2}{p\|m_{t+1}\| + q} +\beta^2 \sum_{t=1}^T \frac{ \|m_{t}\|^2}{p\|m_{t}\| + q} \nonumber \\
     =& \sum_{t=1}^T  \frac{p\|m_{t+1}\|^2(\|m_{t}\| - \|m_{t+1}\|)}{(p\|m_{t}\| + q) (p\|m_{t+1}\|+ q)}  +  \beta^2 \frac{ \|m_{1}\|^2}{p\|m_{1}\| + q} +(\beta^2-1) \sum_{t=2}^{T+1} \frac{\|m_{t}\|^2}{p\|m_{t}\| + q} 
\end{align}
The first equality follows by $\eta_{t} = \tfrac{ 1}{p\|m_{t}\| + q} $. The second equality subtract and add the same terms $\tfrac{\|m_{t+1}\|^2}{p\|m_{t+1}\| + q}$. The last equality regroups the terms.  

We now prove the first term in ~\eqref{eq:term2} admits the following upper bound:
\begin{equation}\label{ineq:term ii}
   \frac{p\|m_{t+1}\|^2(\|m_{t}\| - \|m_{t+1}\|)}{(p\|m_{t}\| + q) (p\|m_{t+1}\|+ q)}  \le (1-\beta) \frac{\|m_{t+1}\|^2}{ p\|m_{t+1}\|+ q} + \frac{(1-\beta)p\| g(y_{t+1})\|}{q}\frac{\|m_{t}\|^2}{ p\|m_{t}\|+ q}
\end{equation}
Note that if $\| m_{t+1} \| \ge \|m_{t}\|$ then the inequality trivially holds. Thus, we only need to consider the case when $\| m_{t+1} \| \le \|m_{t}\|$. By triangle inequality, 
\begin{align*}
    \|m_{t}\| - \|m_{t+1}\| &\le \|m_{t} -  m_{t+1}\| = (1-\beta) \| m_{t} - g(y_{t+1}) \| \\
    & \le (1-\beta)(\|m_{t}\| + \| g(y_{t+1})\|).
\end{align*}
Therefore, substitue the above inequality into lefthand side of ~\eqref{ineq:term ii} and regroup the fractions,
\begin{align*}
 \frac{p\|m_{t+1}\|^2(\|m_{t}\| - \|m_{t+1}\|)}{(p\|m_{t}\| + q) (p\|m_{t+1}\|+ q)} 
 \le  \quad & \frac{p\|m_{t+1}\|^2(1-\beta)(\|m_{t}\| +\| g(y_{t+1})\|)}{(p\|m_{t}\| + q) (p\|m_{t+1}\|+ q)} \\
 = \quad & (1-\beta)\frac{\|m_{t+1}\|^2}{ p\|m_{t+1}\|+ q} \frac{p\|m_{t}\|}{p\|m_{t}\| + q}+ \frac{(1-\beta)p\| g(y_{t+1})\|}{p\|m_{t}\| + q } \frac{\|m_{t+1}\|^2}{p\|m_{t+1}\|+ q} \\
 \le \quad& (1-\beta) \frac{\|m_{t+1}\|^2}{ p\|m_{t+1}\|+ q} + \frac{(1-\beta)p\| g(y_{t+1})\|}{q}\frac{\|m_{t}\|^2}{ p\|m_{t}\|+ q},
\end{align*}
where the last step we use the fact that $\| m_{t+1} \| \le \|m_{t}\|$ and the function $x \rightarrow x^2/(px+q)$ is increasing on $\R_+$. Now, taking expectation on both sides of (\ref{ineq:term ii}) yields 
\begin{align*}
\E\left [\frac{p\|m_{t+1}\|^2(\|m_{t}\| - \|m_{t+1}\|)}{(p\|m_{t}\| + q) (p\|m_{t+1}\|+ q)} \right ] & \le (1-\beta)\E\left [ \frac{\|m_{t+1}\|^2}{ p\|m_{t+1}\|+ q}\right] + \frac{p(1-\beta)}{q} \E \left [ \| g(y_{t+1}) \| \frac{\|m_{t}\|^2}{ p\|m_{t}\|+ q}\right ] \\
& =  (1-\beta)\E\left [ \frac{\|m_{t+1}\|^2}{ p\|m_{t+1}\|+ q}\right] + \frac{p(1-\beta)}{q} \E \left[  \E \left [ \| g(y_{t+1}) \| | \G_t \right] \frac{\|m_{t}\|^2}{ p\|m_{t}\|+ q}\right ] \\
& \le (1-\beta)\E\left [ \frac{\|m_{t+1}\|^2}{ p\|m_{t+1}\|+ q}\right] + \frac{p(1-\beta)G}{q} \E \left [ \frac{\|m_{t}\|^2}{ p\|m_{t}\|+ q}\right ] \\
& \le (1-\beta)\E\left [ \frac{\|m_{t+1}\|^2}{ p\|m_{t+1}\|+ q}\right] + \frac{\beta(1-\beta)}{2} \E \left [ \frac{\|m_{t}\|^2}{ p\|m_{t}\|+ q}\right ] 
\end{align*}
where the third inequality follows by the fact that $\E[\|g(y_{t+1}\| |  \G_t] \le \sqrt{L^2 + \sigma^2}$ and the last inequality follows from our choice of parameters ensuring $pG/q \le \beta/2$.

Now we are ready to proceed the telescopic summing. Summing up over $t$  and  yields
\begin{align*}
     & \sum_{t=1}^{T} \E \left [ \eta_{t}(-\|m_{t+1}\|^2  + \beta^2 \|m_{t}\|^2 ) \right ] \\
    \le \,\, &   (1-\beta) \sum_{t=1}^T  \E \left [\frac{\|m_{t+1}\|^2}{p\|m_{t+1}\| + q} \right ] + \frac{\beta-\beta^2}{2
    } \sum_{t=1}^T \E \left [\frac{\|m_{t}\|^2}{p\|m_{t}\| + q} \right ]+ \beta^2 \E \left [\frac{ \|m_{1}\|^2}{p\|m_{1}\| + q} \right ] + (\beta^2-1) \sum_{t=2}^{T+1} \E \left [\frac{\|m_{t}\|^2}{p\|m_{t}\| + q} \right ]  \\
    = \,\,& \frac{\beta^2+\beta}{2} \E \left [\frac{ \|m_{ 1}\|^2}{p\|m_{1}\| + q} \right ] + \frac{\beta^2-\beta}{2} \sum_{t=2}^{T+1} \E \left [\frac{\|m_{t}\|^2}{p\|m_{t}\| + q} \right ]\\
    =  \,\,& \beta^2 \E \left [\frac{ \|m_{ 1}\|^2}{p\|m_{1}\| + q} \right ] + \frac{\beta^2-\beta}{2} \sum_{t=1}^{T+1} \E \left [\frac{\|m_{t}\|^2}{p\|m_{t}\| + q} \right ]\\
    \le \,\, &  \frac{\beta^2 G^2}{q}+ \frac{\beta^2-\beta}{2} \sum_{t=1}^{T+1} \E \left [\frac{\|m_{t}\|^2}{p\|m_{t}\| + q} \right ]
\end{align*}

The first inequality uses \eqref{ineq:term ii}. The third line  and the foruth line regroup the terms. The last line follows by $p\|m_{1}\| + q \ge q$ and  $\E[\|m_{ 1}\|^2] \le G^2$.

\textbf{Combine all term i, ii and iii in (\ref{eq:=0})} yields
\begin{align*}
    \frac{\beta - \beta^2}{2} \sum_{t=1}^{T+1} \E\left [\frac{\|m_{t}\|^2}{p\|m_{t}\| + q} \right ] &\le 2\beta(1 - \beta)  \E[f(x_{1}) - f(x_{T+1})] + \frac{\beta^2 G^2}{q} 
    + T(1-\beta)^2\frac{G^2
     }{ q}.
\end{align*}

Multiply both side by $\frac{2q}{T(\beta-\beta^2)}$ we get
\begin{align}\label{eq:K}
    \frac{1}{T} \sum_{t=1}^T \E\left [\frac{q\|m_{t}\|^2}{p\|m_{t}\| + q} \right ] &\le \frac{4q \Delta}{T}  + \frac{2\beta G^2}{T(1-\beta)} +  \frac{2(1-\beta)G^2}{\beta}
\end{align}


We may assume $\epsilon \le G$, otherwise any $x_t$ is a $(\delta,\epsilon)$-stationary point. Then by choosing $ \beta = 1 - \frac{\epsilon^2}{64 G^2}$, $p  = \frac{ 64 G^2 \ln(16G/\epsilon)}{\delta \epsilon^2 }$, $q =  \frac{ 256 G^3 \ln(16G/\epsilon)}{\delta \epsilon^2 }$, $T = \frac{2^{16} G^3 \Delta \,\, \text{ln}(16G/\epsilon)}{ \epsilon^4 \delta}\max \{1, \frac{  G \delta}{8\Delta} \}$, have
\begin{align}
    \frac{1}{T} \sum_{t=1}^T \E\left [\frac{4G\|m_{t}\|^2}{\|m_{t}\| + 4G} \right ]  \le \frac{\epsilon^2}{17}
\end{align}
Note that the function $x \rightarrow x^2/(x+4G)$ is convex, 
 thus by Jensen's inequality, for any $t$, we have 
\begin{align}
  \frac{4G  \E\left [\|m_{t}\|\right ]^2}{ \E[\|m_{t}\|] + 4G}  \le \E\left [\frac{4G\|m_{t}\|^2}{\|m_{t}\| + 4G} \right ]
\end{align}
Let's denote 
\[ m_{avg} = \frac{1}{T}\sum_{t=1}^T \E\left [ \|m_{t}\|\right ],\]
then again by Jensen's inequality,
\[ \frac{4G  m_{avg}^2}{m_{avg} + 4G} \le \frac{1}{T} \sum_{t=1}^T  \frac{4G  \E\left [\|m_{t}\|\right ]^2}{ \E[\|m_{t}\|] + 4G} \le \frac{\epsilon^2}{17} \]
Solving the quadratic inequality with respect to $m_{avg}$ and using $\epsilon \le G$, we have
\[  \frac{1}{T}\sum_{t=1}^T \E\left [ \|m_{t}\|\right ] \le \frac{\epsilon}{4}.\]

In contrast to the smooth case, we cannot directly conclude from this inequality since $m_t$ is not the gradient at $x_t$. Indeed, it is the convex combination of all previous stochastic gradients. Therefore, we still need to find a reference point such that $m_t$ is approximately in the $\delta$-subdifferential of the reference point. 
Note that 
\[ m_{t} = \sum_{i=t-K+1}^{t} \alpha_i g(y_i) +  \beta^K m_{t-K}     \]
Intuitively, when $K$ is sufficiently large, the contribution of the last term in $m_t$ is negligible. In which case, we could deduce $m_t$ is approximately in $\partial f(x_{t-K}+\delta B)$. More precisely, with  $\beta = 1 - \frac{\epsilon^2}{64G^2}$, as long as $K \ge  \frac{64G^2}{\epsilon^2} \ln( \frac{16G}{\epsilon})$, we have 
\[ \beta^K \le \frac{\epsilon}{16 G}. \]
This is  a simple analysis result using the fact that $ln(1-x) \le -x$. 
Then by Assumption on the oracle, we know that $\E[g(y_i) | \Y_i] \in \partial f(y_i)$ and $ \| y_i - x_{t-K} \|\le \frac{K}{p} \le \delta$ for any $i \in [t-K+1, t]$. Thus, 
\[  \E[g(y_i) | x_{t-K}] \in \partial f(x_{t-K} + \delta B). \]
Consequently, the convex combination 
\[ \frac{1}{\sum \alpha_i} \sum_{i=t-K+1}^{t} \alpha_i \E[g(y_i) | x_{t-K}] \in \partial f(x_{t-K} + \delta B). \]
Note that   $\sum \alpha_i = 1- \beta^K$, the above inclusion could be rewritten as 
\[ \frac{1}{1-\beta^K} (\E[m_{t} | x_{t-K}] - \beta^K m_{t-K}) \in \partial f(x_{t-K} + \delta B). \] 
This implies that conditioned on $x_{t-K}$
\begin{align*}
     d(0, \partial f(x_{t-K} + \delta B)) \le \frac{1}{1-\beta^K} \left( \|\E[m_{t}\ |\ x_{t-K}]\| + \beta^K \| m_{t-K}\| \right ) \le  \frac{1}{1-\beta^K} \left( \E[\| m_{t}\| | x_{t-K}] + \beta^K \|  m_{t-K} \| \right ).
\end{align*}
Therefore, by taking the expectation, 
\[ \E[d(0, \partial f(x_{t-K} + \delta B))] \le \frac{1}{1-\beta^K} \left( \E[\|m_{t}\|] + \beta^K G \right ) \le \frac{1}{1 - \frac{1}{16}} (\E[\| m_{t} \|] + \frac{\epsilon}{16} )  = \frac{16}{15} \E[\| m_{t} \|] + \frac{\epsilon}{15} .\]
Finally, averaging over $t=1$ to $T$ yields,
\[ \frac{1}{T} \sum_{t=1}^T \E[d(0, \partial f(x_{t-K} + \delta B))] \le \frac{16}{15T} \sum_{t=1}^T \E[\| m_t\|] + \frac{\epsilon}{15} \le \frac{\epsilon}{3} \]
When $t<K$, $\partial f(x_{t-K} + \delta B )$ simply means $\partial f(x_1 + \delta B)$. As a result, if we randomly out put $x_{\max\{1, t-K\}}$ among $t \in [1,T]$, then with at least probability $2/3$, the $\delta$-subdifferential set contains an element with norm smaller than $\epsilon$. To achieve $1-\gamma$ probability result for arbitrary $\gamma$, it suffices to repeat the algorithm $\log (1/\gamma)$ times. 

\end{proof}

\section{Proof of Theorem~\ref{thm:lower}} \label{app:lower}

\begin{figure}[t]
\centering
\includegraphics[width=0.6\linewidth]{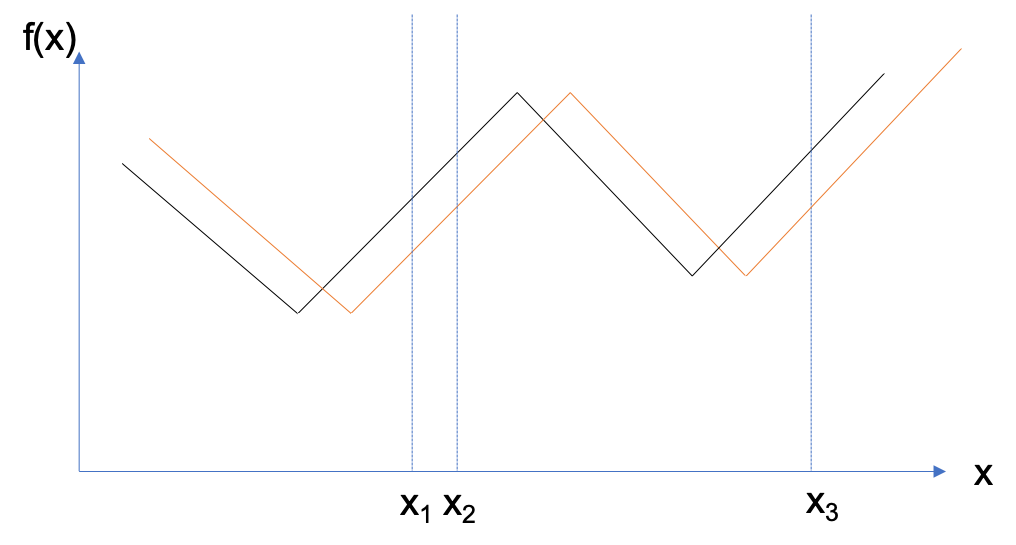}
\caption{}\label{fig:proof}
\end{figure}

\begin{proof}

The proof idea is similar to Proof of Theorem~\ref{thm:finite}. Since the algorithm does not have access to function value, our resisting strategy now always returns
\begin{align*}
    \nabla f(x) = 1.
\end{align*}

If we can prove that for any set of points $x_{k}, k \in [1,K], K\le \frac{\Delta}{8\delta}$, there exists two one dimensional functions such that they satisfy the resisting strategy $\nabla f(x_k) = 1, k \in [1,K]$, and that the two functions do not have two stationary points that are $\delta$ close to each other, then we know no randomized/deterministic can return an $(\delta, \epsilon)-$stationary points with probability more than $1/2$ for both functions simultaneously. In other word, no algorithm that query $K$ points can distinguish these two functions. Hence we proved the theorem following the definition of complexity in~\eqref{eq:complexity}. 

From now on, let $x_{k}, k \in [1,K]$ be the sequence of points queried after sorting in ascending order. Below, we construct two functions such that $\nabla f(x_k) = 1, k \in [1,K]$, and that the two functions do not have two stationary points that are $\delta$ close to each other. Assume WLOG that $x_k$ are ascending. First, we define $f:\mathbb{R} \to \mathbb{R}$ as follows:
\begin{align*}
    &f(x_0) = 0,\\
    &f'(x) = -1 \quad \text{if  } \quad  x  \le x_1 - 2\delta, \\
    &f'(x) = 1 \quad \text{if exists $i \in [K]$ such that} \quad  |x - x_i| \le 2\delta, \\
    &f'(x) = -1 \quad \text{if exists $i \in [K]$ such that} \quad  x \in [x_i + 2\delta, \frac{x_i+x_{i+1}}{2}],\\
    &f'(x) = 1 \quad \text{if exists $i \in [K]$ such that} \quad  x \in [ \frac{x_i+x_{i+1}}{2}, x_{i+1} - 2\delta],\\
    &f'(x) = 1 \quad \text{if  } \quad  x  \ge x_K + 2\delta 
\end{align*}
A schematic picture is shown in Figure~\ref{fig:proof}. It is clear that this function satisfies the resisting strategy. It also has stationary points that are at least $4\delta$ apart. Therefore, simply by shifting the function by $1.5\delta$, we get the second function.

The only thing left to check is that $\sup_k f(x_k) - \inf_x f(x) \le \Delta$. By construction, we note that the value from $x_i$ to $x_{i+1}$ is non decreasing and increase by at most $4\delta$
\begin{align}\label{eq:subopt-in-lower}
    \sup_k f(x_k) - f(x_0) \le 4\delta K \le \Delta/2.
\end{align}
We further notice that the global minimum of the function is achieved at $x_0 - 2\delta$, and $f(x_0 - 2\delta) = -2\delta \le 4\delta K \le \Delta/2.$ Combined with~\eqref{eq:subopt-in-lower}, we get,
\begin{align}
    \sup_k f(x_k) - \inf_x f(x) \le \Delta.
\end{align}

\end{proof}

\end{document}